\documentclass[leqno]{article}

\usepackage[frenchb,english]{babel}
\usepackage[utf8]{inputenc}
\usepackage{amsmath}
\usepackage{amssymb}
\usepackage{amsfonts}
\usepackage{enumerate}
\usepackage{vmargin}
\usepackage[all]{xy}
\usepackage{mathrsfs}
\usepackage{mathtools}
\usepackage{lmodern}
\usepackage{slashed}
\usepackage[colorlinks=true,linkcolor=blue,pagebackref=true]{hyperref}%
\setmarginsrb{3cm}{3cm}{3.5cm}{3cm}{0cm}{0cm}{1.5cm}{3cm}
\usepackage{comment}

\usepackage{accents}

\newlength{\dhatheight}

\newcommand{\E}{\ensuremath{\mathbb{E}}}

\newcommand{\B}{\mathrm{B}} 

\let\L\relax 
\newcommand{\L}{\mathrm{L}}

\newcommand{\Mult}{\mathrm{Mult}}

\newcommand{\M}{\mathrm{M}}

\let\cal\relax
\newcommand{\cal}{\mathcal}

\newcommand{\Zc}{\mathrm{Z}}

\newcommand{\W}{\mathrm{W}}

\newcommand{\Id}{\mathrm{Id}}
\newcommand{\la}{\langle}
\newcommand{\ra}{\rangle}

\renewcommand{\leq}{\ensuremath{\leqslant}}
\renewcommand{\geq}{\ensuremath{\geqslant}}
\newcommand{\qed}{\hfill \vrule height6pt  width6pt depth0pt}

\newcommand{\norm}[1]{\left\Vert#1\right\Vert}

\newcommand{\co}{\colon}

\newcommand{\ot}{\otimes}
\newcommand{\ovl}{\overline}
\newcommand{\otvn}{\ovl\ot}

\newcommand{\dec}{\mathrm{dec}}

\newcommand{\TRO}{\mathrm{TRO}}

\newcommand{\ov}{\overset}

\newcommand{\w}{\mathrm{w}} 

\DeclareMathOperator{\Span}{span} 
\let\ker\relax 
\DeclareMathOperator{\ker}{Ker} 
\DeclareMathOperator{\Ran}{Ran} 

\newtheorem{thm}{Theorem}[section]

\newtheorem{prop}[thm]{Proposition}

\newtheorem{cor}[thm]{Corollary}
\newtheorem{lemma}[thm]{Lemma}

\newtheorem{remark}[thm]{Remark}
\newtheorem{example}[thm]{Example}

\newenvironment{proof}[1][]{\noindent {\it Proof #1} : }{\hbox{~}\qed
\smallskip
}

\usepackage{tocloft}
\numberwithin{equation}{section}
\usepackage[nottoc,notlot,notlof]{tocbibind}

\let\OLDthebibliography\thebibliography
\renewcommand\thebibliography[1]{
  \OLDthebibliography{#1}
  \setlength{\parskip}{0pt}
  \setlength{\itemsep}{0pt plus 0.3ex}
}

\newcommand\reallywidehat[1]{\arraycolsep=0pt\relax%
\begin{array}{c}
\stretchto{
  \scaleto{
    \scalerel*[\widthof{\ensuremath{#1}}]{\kern-.5pt\bigwedge\kern-.5pt}
    {\rule[-\textheight/2]{1ex}{\textheight}} 
  }{\textheight} %
}{0.5ex}\\           
#1\\                 
\rule{-1ex}{0ex}
\end{array}
}

\begin{document}
\selectlanguage{english}
\title{\bfseries{Completely isometric subspaces of noncommutative $\L^p$-spaces and contractive projections}}
\date{}
\author{\bfseries{C\'edric Arhancet}}%
\maketitle


\begin{abstract}
We investigate the relation between the complete isometry class of subspaces of noncommutative $\mathrm{L}^p$-spaces and their contractive complementability, where $1 < p < \infty$ with $p \not= 2$. We show that if $P \co \mathrm{L}^p(\mathcal{M}) \to \mathrm{L}^p(\mathcal{M})$ is a positive contractive projection whose range is completely isometric to another noncommutative $\mathrm{L}^p$-space, then $P$ is necessarily completely positive. This provides a converse to the main result of \cite{ArR24} and the positivity assumption on $P$ is essential. We further establish a rectangular analogue of this result. More precisely, we prove that every closed subspace of a noncommutative $\mathrm{L}^p$-space which is completely isometric to a rectangular noncommutative $\mathrm{L}^p$-space of the form $e\mathrm{L}^p(\mathcal{N})(1-e)$ is the range of a contractively decomposable projection. Combined with the known converse implication, this yields a characterization of the ranges of contractively decomposable projections as precisely the subspaces completely isometric to rectangular $\mathrm{L}^p$-spaces associated with $\mathrm{W}^*$-ternary rings of operators.
\end{abstract}


\makeatletter
 \renewcommand{\@makefntext}[1]{#1}
 \makeatother

{
  \hypersetup{linkcolor=blue}
 \tableofcontents
}

\section{Introduction}
\label{sec:Introduction}

The study of contractive projections and contractively complemented subspaces is a classical topic in Banach space theory. A fundamental result of Douglas and Ando says that, for $1\leq p<\infty$ with $p\not=2$, a closed subspace of a classical $\L^p$-space is contractively complemented if and only if it is isometrically isomorphic to an $\L^p$-space, see \cite{Dou65,And66} and also \cite{Tza69,BeL74} for extensions to more general measure spaces. Thus, in the commutative setting, the isometric structure of a subspace determines whether it is the range of a contractive projection. Since $\L^p$-spaces are smooth for $1 < p < \infty$, such a projection is in addition uniquely determined by its range by \cite[Theorem~6]{CoS70} (see Proposition \ref{prop-Sullivan}).

The corresponding problem for noncommutative $\L^p$-spaces is considerably more involved. Even for Schatten spaces, the ranges of contractive projections need not be isometric to noncommutative $\L^p$-spaces. The work of Arazy and Friedman \cite{ArF78,ArF92} gives a complete description of contractively complemented subspaces of Schatten spaces in terms of Cartan factors. From the operator space point of view, a more rigid picture emerges when matricial assumptions are imposed. In particular, Le Merdy, Ricard and Roydor characterized completely contractively complemented subspaces of Schatten spaces in \cite{LRR09}. For $p=1$, Ng and Ozawa proved a particularly suggestive result: a subspace of the predual of a $\mathrm{W}^*$-TRO is completely contractively complemented if and only if it is completely isometric to the predual of a $\mathrm{W}^*$-$\TRO$, see \cite{NgO02}. Recall that a $\mathrm{W}^*$-ternary ring of operators, or $\mathrm{W}^*$-$\TRO$, is a weak* closed operator space $V$ which is closed under the ternary product $(x,y,z)\mapsto xy^*z$. The class of preduals of a $\mathrm{W}^*$-$\TRO$ includes the preduals of von Neumann algebras.

For general noncommutative $\L^p$-spaces, the relation between complete isometry and positive contractive complementability is more subtle. By the classification theorem of Junge, Ruan and Sherman \cite{JRS05}, a complete isometry between noncommutative $\L^p$-spaces is governed by a normal $*$-homomorphism together with a partial isometry. This additional partial-isometry factor shows that the complete isometry class of a subspace does not, by itself, retain its order structure. Nevertheless, we show that this obstruction disappears in the presence of a positive contractive projection. More precisely, if $P\co\L^p(\cal M)\to\L^p(\cal M)$ is a positive contractive projection and $\Ran P$ is completely isometric to a noncommutative $\L^p$-space, then $P$ is necessarily completely positive. Thus, among positively contractively complemented subspaces, being completely isometric to a noncommutative $\L^p$-space already forces the full completely positive structure. This provides a converse to the main result of \cite{ArR24}. The proof is not a formal consequence of the classification of complete isometries: it requires showing that the partial-isometry factor occurring in the Junge--Ruan--Sherman representation can be absorbed into the von Neumann algebraic part of the factorization. We refer to \cite{Arh26b} and \cite{Bof25} for more information on positive contractive projections and to \cite{Arh24a} and \cite{ArL26} for related papers.

The second main purpose of the present paper is to establish a rectangular counterpart of this phenomenon. Note that a $\mathrm{W}^*$-$\TRO$ $V$ can be realized as an off-diagonal corner $e\mathrm{R}(V)e^\perp$ of its linking von Neumann algebra $\mathrm{R}(V)$, where we use the orthogonal projection $e^\perp \ov{\mathrm{def}}{=} 1-e$. In \cite{Arh24b}, if $\mathrm{R}(V)$ is $\sigma$-finite, we associated with a $\mathrm{W}^*$-TRO $V$ a rectangular noncommutative $\L^p$-space
\[
\L^p(V,\varphi)
\ov{\mathrm{def}}{=} e\L^p(\mathrm{R}(V),\varphi)e^\perp,
\]
where $\varphi$ is a normal faithful state on the linking von Neumann algebra $\mathrm{R}(V)$ and where $1 \leq p < \infty$. We also showed in \cite{Arh24b} that these rectangular $\L^p$-spaces arise as ranges of contractively decomposable projections on noncommutative $\L^p$-spaces. More recently, the converse structural statement is proved in \cite[Theorem 1.3]{Bof26} : the range of any contractively decomposable projection on a $\sigma$-finite noncommutative $\L^p$-space is completely isometrically isomorphic to a corner $e\L^p(\cal{N})(1-e)$.

The new point of this paper is that the manner in which such a rectangular space is embedded does not matter. More precisely, suppose that $1<p<\infty$ with $p\not=2$ and let $Y$ be a closed subspace of $\L^p(\cal M)$, where $\cal{M}$ is $\sigma$-finite. We prove that if $Y$ is completely isometric to a rectangular noncommutative $\L^p$-space $e\L^p(\cal{N})e^\perp$, then $Y$ is the range of a contractively decomposable projection on $\L^p(\cal{M})$. Combined with \cite[Theorem~1.3]{Bof26}, this implies that $Y$ is contractively decomposably complemented if and only if $Y$ is completely isometric to some corner $e\L^p(\cal{N})e^\perp$. Equivalently, the ranges of contractively decomposable projections are precisely the subspaces completely isometric to rectangular $\L^p$-spaces associated with $\mathrm{W}^*$-TROs.

\paragraph{Approach of the paper} 
The proof of the first main result relies on the classification of complete isometries of Junge, Ruan and Sherman \cite{JRS05}. Let $P\co\L^p(\cal M)\to\L^p(\cal M)$ be a positive contractive projection and suppose that $\Ran P$ is completely isometric to $\L^p(\cal N)$. A complete isometry $T\co\L^p(\cal N)\to\Ran P$ admits a factorization of the form $T=w\widetilde T$, where $w$ is a partial isometry and $\widetilde T$ is a completely positive complete isometry associated with a normal $*$-homomorphism. The main point is to show that the partial-isometry factor $w$, which in general prevents $T$ from being positive, can be absorbed into the von Neumann algebraic part of the factorization. We first use the positivity of $P$, hence the selfadjointness of its range, together with left and right support projections, to show that $w$ is a unitary in the relevant corner. We then consider a positive element of $\Ran P$ having full support and use the bimodule structure of $\widetilde T$ to prove that $w$ belongs to the von Neumann algebra generated by the image of the underlying $*$-homomorphism. It follows that $\Ran T=\Ran\widetilde T$. Since the latter range is completely positively and completely contractively complemented by the Junge--Ruan--Sherman theorem, the uniqueness of a contractive projection onto a fixed subspace of the smooth Banach space $\L^p(\cal M)$ implies that $P\co \L^p(\cal M)\to\L^p(\cal M)$ itself is completely positive.

The proof of the second main result is based on a stabilization argument. After amplification by $S^p$, a rectangular noncommutative $\L^p$-space can be identified completely isometrically with an ordinary noncommutative $\L^p$-space. The classification theorem of Junge, Ruan and Sherman then yields a completely contractive projection onto the stabilized copy. The particular form of this projection shows that it is in fact contractively decomposable. By compressing to a matrix corner, we obtain a contractively decomposable projection onto the original subspace. 


\paragraph{Structure of the paper}
The paper is organized as follows. In Section~\ref{sec-preliminaries}, we recall some facts on $\mathrm{W}^*$-TROs, support projections, properly infinite projections and uniqueness of contractive projections. In Section~\ref{sec-cp}, we discuss completely isometric copies of noncommutative $\L^p$-spaces and their complementability. In Section~\ref{sec-rectangular-spaces}, we consider rectangular noncommutative $\L^p$-spaces and prove the characterization of the ranges of contractively decomposable projections described previously.

\section{Preliminaries}
\label{sec-preliminaries}

This section collects the few structural facts needed in the proofs of our two main results. Properly infinite projections will be used in the stabilization argument of Section~\ref{sec-rectangular-spaces}. The language of $\W^*$-$\TRO$s and linking von Neumann algebras provides the natural framework for rectangular noncommutative $\L^p$-spaces. Support projections will be used to analyze the partial isometry arising in the Junge--Ruan--Sherman factorization, whereas the Cohen--Sullivan uniqueness theorem will allow us to identify contractive projections having the same range.

\paragraph{Projections in von Neumann algebras}
Let $p$ be an orthogonal projection in a von Neumann algebra $\cal{M}$. The smallest projection $c(p)$ in the center $\Zc(\cal{M})$ containing $p$ as a subprojection is called the central support (or central carrier) \cite[III.1.1.5 p.~223]{Bla06} (see also \cite[Definition 1.5.7 p.~34]{Li92}) of $p$. By \cite[Proposition 1.5.8 p.~34]{Li92}, if $p$ is a projection in $\cal{M}$, and $z$ is a central projection in $\cal{M}$. Then 
\begin{equation}
\label{support-central-et-z}
zc(p) 
= c(zp).
\end{equation}
Following \cite[Definition 6.1.4 p.~402]{KaR97b}, two orthogonal projections $p$ and $q$ in a von Neumann algebra $\cal{M}$ are said to be (Murray-von Neumann) equivalent in $\cal{M}$, written $p \sim q$, if there exists a partial isometry $w \in \cal{M}$ such that $w^*w = p$ and $ww^* = q$. We say that $p$ is subordinate to $q$, denoted by $p \lesssim q$ if $p$ is equivalent to a subprojection of $q$, i.e., if there is a partial isometry $w \in \cal{M}$ with $w^*w = p$ and $ww^* \leq q$.  

\paragraph{Properly infinite projections}
An orthogonal projection $p$ in a von Neumann algebra $\cal{M}$ is called properly infinite if there exist some orthogonal projections $p_1$ and $p_2$ such that $p \sim p_1$, $p \sim p_2$, $p_1 \leq p$, $p_2 \leq p$ and $p_1 \perp p_2$. Roughly speaking, a properly infinite projection contains two orthogonal copies of itself.

Following \cite[p.~226]{Bla06}, a von Neumann algebra $\cal{M}$ is said to be locally countably decomposable if there is a family $(z_i)_{i \in I}$ of mutually orthogonal central projections in $\cal{M}$ with $\sum_{i \in I} z_i = 1_{\cal{M}}$ such that each von Neumann algebra $\cal{M}z_i$ is countably decomposable. The following is \cite[Corollary III.1.3.7 p.~229]{Bla06}.

\begin{prop}
\label{prop-Blackadar}
Let $\cal{M}$ be a locally countably decomposable von Neumann algebra, and $p$ a properly infinite projection in $\cal{M}$ with central support 1. Then $p \sim 1$.
\end{prop}

\paragraph{Ternary rings of operators} Rectangular operator spaces are naturally modeled by off-diagonal corners of von Neumann algebras. The intrinsic objects corresponding to such corners are ternary rings of operators, which we briefly recall. A ternary ring of operators (or simply $\TRO$) is a norm closed subspace $V$ of the space $\B(H,K)$, for some complex Hilbert spaces $H$ and $K$, which is closed under the triple product $(x,y,z) \mapsto xy^*z$, see, e.g., \cite[4.4.1 p.~161]{BLM04}. A TRO $V$ is called a $\W^*$-$\TRO$ if it is weak* closed in the dual Banach space $\B(H,K)$. A sub-TRO of a TRO $V$ is a closed subspace $W$ of $V$ satisfying $W W^*W \subset W$. We refer to \cite{BLM04}, \cite{BFT12}, \cite{BuT19}, \cite{BuT13a}, \cite{BuT13b}, \cite{DoR07}, \cite{EsM23}, \cite{Ham92}, \cite{Ham99}, \cite {KaR02}, \cite{NeR03}, \cite{PlR19}, \cite{SaS13}, \cite{SaS17} and \cite{Zet83} for more information on TROs.

\begin{example} \normalfont
\label{typical-TRO}
A basic example of $\W^*$-$\TRO$ is given by $e\cal{M}f$ where $e,f$ are orthogonal projections of a von Neumann algebra $\cal{M}$.
\end{example}

\paragraph{Operator space structures of TROs}
Since the results of this paper are formulated in terms of complete isometries, the canonical operator space structure of a TRO will play an essential role. Every TRO admits an operator space structure. Indeed, let us assume that $V$ is a TRO contained in $\B(H,K)$. Then for each integer $n \geq 1$, the matrix space $\M_n(V)$ can be identified with a TRO contained in $\M_n(\B(H,K)) \cong \B(H^n,K^n)$. This provides a canonical operator space matrix norm on $V$ such that each matrix space $\M_n(V)$ is a TRO. By \cite[Proposition 2.1 p.~265]{KaR02} or \cite[Proposition 2.1 p.~83]{Ham99}, the TRO-matrix norms are uniquely determined on each TRO and do not depend on the choice of the representing Hilbert spaces.

\paragraph{$\TRO$-homomorphisms} A $\TRO$-homomorphism (or triple morphism) is a linear map $T \co V \to W$ between two $\TRO$s respecting the triple product, i.e.,
$$
T(xy^*z) 
= T(x)T(y)^*T(z), \quad x,y,z \in V.
$$
If, in addition, $T$ is an injection from $V$ onto $W$, we call $T$ a $\TRO$-isomorphism from $V$ onto $W$. By \cite[Proposition 2.4 p.~498]{EOR01}, a TRO-isomorphism between $\W^*$-TROs is necessarily weak* continuous. By  \cite[Proposition 2.1 p.~495]{EOR01}, every TRO-homomorphism is completely contractive, and every injective $\TRO$-homomorphism is completely isometric. Finally, by \cite[Proposition 2.1 p.~83]{Ham99} (see also \cite[Corollary 4.4.6 p.~163]{BLM04}), if $T \co V \to W$ is a linear map between TRO's then $T$ is a surjective complete isometry if and only if $T$ is a surjective 2-isometry if and only if $T$ is a TRO-isomorphism. Thus, for TROs, the complete isometric structure already determines the triple product.

\begin{example} \normalfont
\label{E-finite-dim-TRO}
Every finite-dimensional $\TRO$ $V$ is completely isometric to a finite direct sum of rectangular matrix algebras, i.e., it has the form
$$
V 
= \M_{m_1,n_1} \oplus_\infty \cdots \oplus_\infty \M_{m_k,n_k}, 
$$
see \cite[Corollary A.2 p.~302]{Kan13}.
\end{example}


\paragraph{Linking algebra} The linking algebra allows us to pass back and forth between the rectangular setting of TROs and the usual setting of von Neumann algebras. If $\cal{M}$ is a von Neumann algebra and $e$ is an orthogonal projection in $\cal{M}$, then $e\cal{M} e^\perp$ is a $\W^*$-TRO by Example \ref{typical-TRO}. Conversely, if the subspace $V$ of the space $\B(H,K)$ is a $\W^*$-TRO, then we can consider the adjoint space $V^* \ov{\mathrm{def}}{=} \{x^* : x \in V \}$, which is a subspace of the space $\B(K,H)$ and the von Neumann subalgebras 
\begin{equation}
\label{def-MV-NV}
\mathrm{M}(V) \ov{\mathrm{def}}{=} \ovl{\Span VV^*}^{\w^*}
\quad \text{and} \quad 
\mathrm{N}(V) \ov{\mathrm{def}}{=} \ovl{\Span V^*V}^{\w^*}
\end{equation}
of the algebras $\B(K)$ and $\B(H)$. Since $\mathrm{M}(V) \cdot V \subset V$ and $V \cdot \mathrm{N}(V) \subset V$, we can introduce the von Neumann algebra
\begin{equation}
\label{Linking-algebra}
\mathrm{R}(V)
\ov{\mathrm{def}}{=} \begin{bmatrix}
  \mathrm{M}(V)   &  V \\
  V^*   &  \mathrm{N}(V) \\
\end{bmatrix}
\subset \B(K \oplus H),
\end{equation}
which is called the linking von Neumann algebra of $V$, see \cite{KaR02}, \cite[p.~846]{Rua04} and \cite{WCW24}. Then we have a TRO-isomorphism 
\begin{equation}
\label{TRO-linking}
V 
= e \mathrm{R}(V)e^\perp,
\end{equation}
where $e 
\ov{\mathrm{def}}{=} \begin{bmatrix}
    \Id_K & 0  \\
    0 & 0  \\
\end{bmatrix}
$ 
and 
$e^\perp
\ov{\mathrm{def}}{=} \begin{bmatrix}
   0  &  0 \\
   0  &  \Id_H \\
\end{bmatrix}$. So a $\W^*$-TRO $V$ can be identified with the off-diagonal corner at the (1,2) position of its linking von Neumann algebra. By \cite[Lemma 2.1 p.~847]{Rua04}, the central covers $c(e)$ and $c(e^\perp)$ of $e$ and $e^\perp$ in the linking von Neumann algebra $\mathrm{R}(V)$ are equal to $1$.

Conversely, it is worth noting that, according to \cite[Theorem 2.1]{WCW24}, a von Neumann algebra is $*$-isomorphic to the linking von Neumann algebra of a $\W^*$-TRO if and only if it contains no abelian direct summand.

\paragraph{Support projections}
We next recall the support notation needed in Section~\ref{sec-cp}. If $x$ is an unbounded operator on a Hilbert space,  we have
\begin{equation}
\label{support-projections}
s_r(x)
=s(x^*x)
\quad \text{and} \quad
s_\ell(x)
=s(xx^*).
\end{equation}
If $x=u|x|$ is the polar decomposition of $x$, then $s_r(x)=u^*u$ and $s_\ell(x)=uu^*$. Thus the right and left supports are respectively the initial and final projections of the partial isometry $u$ occurring in the polar decomposition.

\paragraph{Duality mappings}

Following \cite[Definition 5.1.1 p.~426]{Meg98}, we say that a normed linear space $X$ is strictly convex (or rotund) if for any $t \in (0,1)$ and any $x,y \in X$ with $\norm{x}_X=\norm{y}_X=1$ and $x \not= y$ we have $\norm{tx+(1-t)y}_X <1$. 
A normed space $X$ is said to be smooth \cite[p.~480]{Meg98} if for any $x \in X$ there exists a unique $x^* \in X^*$ with $\norm{x^*}_{X^*} = 1$ such that $\la x^*,x\ra=1$. According to \cite[Proposition 5.4.7 p.~481]{Meg98}, a \textit{reflexive} normed space is smooth if and only if its dual space is strictly convex. 

We refer to \cite{Pat18} and \cite{Cio90} for more information on  duality mappings and their relation with the geometry of Banach spaces. Let $X$ be a Banach space. For each $x \in X$, we define the subset
\begin{equation}
\label{Def-JX-gauge}
\frak{j}(x)
\ov{\mathrm{def}}{=}
\big\{x^*\in X^*:\la x,x^*\ra_{X,X^*}
=\norm{x}_X^2,
\norm{x^*}_{X^*}=\norm{x}_X\big\}.
\end{equation}
of the dual $X^*$. The multivalued map $\frak{j} \co X \to 2^{X^*}$ is called the duality mapping. By the Hahn--Banach theorem\footnote{\thefootnote. For any $x \in X$, there exists $y^* \in X^*$ with $\norm{y^*}_{X^*} = 1$ such that $\la x, y^*\rangle_{X,X^*} = \norm{x}_X$. Using $x^*=\norm{x}_Xy^*$, we conclude that $\frak{j}(x) \not= \emptyset$ for each $x \in X$.}, $\frak{j}(x)$ is nonempty for every $x \in X$, see \cite[Remark 4.2 p.~25]{Cio90}. The Banach space $X$ is smooth if and only if $\frak{j}$ is single-valued, see \cite[Corollary 4.5 p.~27]{Cio90}. 

We will use the following observation \cite[Theorem 6]{CoS70} of Cohen and Sullivan. We include the short uniqueness argument. For any $x \in X$, let $\frak{j}(x)$ denote its unique normalized norming functional, so that 
\begin{equation}
\label{normalized-norming-functional}
\norm{\frak{j}(x)}_{X^*}
= \norm{x}_{X}
\quad \text{and} \quad
\la x,\frak{j}(x) \ra_{X,X^*} 
= \norm{x}_{X}^2.
\end{equation}

\begin{prop}[Cohen--Sullivan]
\label{prop-Sullivan}
A subspace $Y$ of a smooth Banach space $X$ can be the range of at most one projection of norm one. 
\end{prop}

\begin{proof}
If $R \co X \to X$ is any contractive projection onto the subspace $Y$ then, for any $y \in Y$, we have
\[
\big\la y,R^*\frak{j}(y) \big\ra_{X,X^*}
=\big\la R(y),\frak{j}(y) \big\ra_{X,X^*}
=\la y,\frak{j}(y) \ra_{X,X^*}
\ov{\eqref{normalized-norming-functional}}{=} \norm{y}_{X}^2
\]
whereas $\norm{R^*\frak{j}(y)}_{X^*} \leq \norm{\frak{j}(y)}_{X^*} \ov{\eqref{normalized-norming-functional}}{=} \norm{y}_{X}$. By uniqueness of the norming functional, $R^*\frak{j}(y) = \frak{j}(y)$. Suppose that $P,Q \co X \to X$ are contractive projections on $Y$. Now, we apply the previous observation to both $P$ and $Q$. For any $x \in X$, introducing the element $y \ov{\mathrm{def}}{=} P(x)-Q(x)$ in $Y$, we obtain
\begin{align*}
\MoveEqLeft
\norm{y}_{X}^2
\ov{\eqref{normalized-norming-functional}}{=} \la y,\frak{j}(y) \ra_{X,X^*} 
= \la P(x)-Q(x),\frak{j}(y) \ra_{X,X^*} 
=\la (P-Q)(x),\frak{j}(y) \ra_{X,X^*} \\
&=\la x,P^*\frak{j}(y)-Q^*\frak{j}(y) \ra_{X,X^*} 
=\la x,\frak{j}(y)-\frak{j}(y) \ra_{X,X^*} 
=0.         
\end{align*}
Consequently, we obtain $y=0$, so $P(x) = Q(x)$ for any $x \in X$. We conclude that $P=Q$.
\end{proof}

\section{Subspaces completely isometric to a noncommutative $\L^p$-space}
\label{sec-cp}


The purpose of this section is to prove that the complete isometric structure of the range of a positive contractive projection detects complete positivity. The main difficulty is that a general complete isometry need not be positive: in the Junge--Ruan--Sherman factorization it contains a left multiplication by a partial isometry. Our argument consists in showing that the positivity of the projection forces this partial isometry to belong to the von Neumann algebraic part of the factorization, where it can be absorbed. We refer to \cite{PiX03} for more information on noncommutative $\L^p$-spaces. We denote by $h_\varphi$ the density operator of a normal linear positive functional $\varphi$. 

We first isolate the positive part of the Junge--Ruan--Sherman representation. Although the complete isometry considered later will not be assumed positive, after removing its partial-isometry factor we obtain a completely positive complete isometry to which the following facts apply.

\paragraph{Positive 2-isometries} Suppose that $1<p<\infty$ with $p\not=2$. Let $\cal{M}$ be a $\sigma$-finite von Neumann algebra equipped with a normal faithful state. Let $\cal{N}$ be a $\sigma$-finite von Neumann algebra equipped with a normal faithful state $\varphi$ and let $T \co \L^p(\cal{N}) \to \L^p(\cal{M})$ be a positive map. According to \cite[Theorem~2 p.~287]{JRS05}, \cite[Theorem~4.9 p.~308]{JRS05} and \cite[Remark~2 p.~310]{JRS05}, the map $\Id\ot T\co\L^p(\M_2\otvn\cal{N})\to\L^p(\M_2\otvn\cal{M})$ is an isometry if and only if there exist an injective normal $*$-homomorphism $\pi \co \cal{N} \to \cal{M}$ and a faithful normal conditional expectation $F \co e\cal{M}e \to \pi(\cal{N})$, where $e \ov{\mathrm{def}}{=} \pi(1)$, such that, if we define $\E \co \cal{M} \to \pi(\cal{N})$ by
\begin{equation}
\label{JRS-expectation-corner}
\E(x)
\ov{\mathrm{def}}{=}F(exe),
\end{equation}
then
\begin{equation}
\label{JRS}
T\big(h_\varphi^{\frac1p}x\big)
=h_{\varphi\circ\pi^{-1}\circ\E}^{\frac1p}\pi(x),\quad x\in\cal{N}.
\end{equation}
Strictly speaking, the map $\E$ is not a conditional expectation from $\cal{M}$ onto $\pi(\cal{N})$, since in general $\E(1)=e=\pi(1)\not=1_{\cal{M}}$. We nevertheless keep the misleading notation and terminology of \cite{JRS05}, where $\E$ is referred to as a conditional expectation onto $\pi(\cal{N})$.

In this case, $T$ is automatically completely positive and completely isometric, and its range is the range of a completely positive and completely contractive projection on $\L^p(\cal{M})$. Moreover, $T$ is a right $\pi(\cal{N})$-module map:
\begin{equation}
\label{module-map}
T(hx)=T(h)\pi(x),\quad h\in\L^p(\cal{N}),\ x\in\cal{N}.
\end{equation}
Finally, by \cite[Theorem~2, (1.2) p.~288]{JRS05}, for every $\omega\in\cal{N}_*^+$ we have
\begin{equation}
\label{JRS-bis}
T\big(h_\omega^{\frac1p}\big)
=h_{\omega\circ\pi^{-1}\circ\E}^{\frac1p}.
\end{equation}

We first record a support property which will be used repeatedly below. It shows that a completely positive complete isometry transports the support projections exactly through the underlying $*$-homomorphism.

\begin{prop}
\label{prop-supports}
Let $\cal{N}$ and $\cal{M}$ be $\sigma$-finite von Neumann algebras equipped with a normal faithful states.  Suppose that $1 < p < \infty$ with $p \not=2$. Let $T \co \L^p(\cal{N}) \to \L^p(\cal{M})$ be a completely positive complete isometry. If $a \in \L^p(\cal{N})_+$ then
\begin{equation}
\label{support-S-pi}
s(T(a))
=\pi(s(a)),
\end{equation}
where $\pi \co \cal{N} \to \cal{M}$ is the normal injective $*$-homomorphism provided by the Junge–Ruan–Sherman factorization of $T$.
\end{prop}

\begin{proof}
Let $\omega \in \cal{N}_*^+$ be the normal positive functional corresponding to the positive element $a^p$ of $\L^1(\cal{N})$. We have $h_\omega=a^p$, $a=h_\omega^{\frac1p}$ and $s(\omega) = s(a^p)=s(a)$. The Junge--Ruan--Sherman formula \eqref{JRS-bis} gives
\[
T(a)
\ov{\eqref{JRS-bis}}{=} h_{\omega \circ \pi^{-1} \circ \E}^{\frac1p}.
\]
Consider the normal linear positive functional
\begin{equation}
\label{def-theta}
\theta 
\ov{\mathrm{def}}{=} \omega \circ \pi^{-1} \circ \E \co \cal{M} \to \mathbb{C}.
\end{equation} 
Since $\E$ is a conditional expectation onto $\pi(\cal N)$, we have
\begin{equation}
\label{inf-39}
\E(\pi(1)-\pi(s(a)))
=\E(\pi(1-s(a)))
=\pi(1-s(a)).
\end{equation}
Consequently, we obtain
\begin{equation}
\label{theta-789}
\theta(\pi(1)-\pi(s(a)))
\ov{\eqref{def-theta}}{=} \omega(\pi^{-1}(\E(\pi(1)-\pi(s(a)))))
\ov{\eqref{inf-39}}{=} \omega(1-s(a))
=\omega(1-s(\omega))
=0.
\end{equation}
Put $e\ov{\mathrm{def}}{=}\pi(1)$. By \eqref{JRS-expectation-corner}, the restriction $\E|_{e\cal{M}e}=F$ is faithful and $\E(x)=\E(exe)$ for every $x \in \cal{M}$. Hence $\E(1-e)=0$ and therefore
\begin{equation}
\label{theta-outside-e}
\theta(1-e)
\ov{\eqref{def-theta}}{=}\omega\circ\pi^{-1}\circ\E(1-e)
=0.
\end{equation}
Consequently, we have
\[
\theta(1-\pi(s(a)))
=\theta(1-\pi(1)) + \theta(\pi(1)-\pi(s(a)))
 \ov{\eqref{theta-outside-e} \eqref{theta-789}}{=} 0,
\]
and therefore $s(\theta)\leq\pi(s(a))$. 

Conversely, consider a positive element $x\in\pi(s(a))\cal{M}\pi(s(a))$ and suppose that $\theta(x)=0$. Since $\E$ is a conditional expectation onto $\pi(\cal{N})$, it is $\pi(\cal{N})$-bimodular. Hence
\[
\E(x)
=\E(\pi(s(a))x\pi(s(a)))
=\pi(s(a))\E(x)\pi(s(a)).
\]
Thus $\E(x)$ is a positive element of the von Neumann algebra $\pi(s(a))\pi(\cal{N})\pi(s(a))$. Since the functional $\omega$ is faithful on $s(a)\cal{N}s(a)$, the equality
\[
0
=\theta(x)
\ov{\eqref{def-theta}}{=}
\omega(\pi^{-1}(\E(x)))
\]
implies $\pi^{-1}(\E(x))=0$, and hence $\E(x)=0$. Moreover, $\pi(s(a))\leq e$, so $x\in e\cal{M}e$. Since the restriction $\E|_{e\cal{M}e}=F$ is faithful by \eqref{JRS-expectation-corner}, we obtain $x=0$. Consequently, $\theta$ is faithful on $\pi(s(a))\cal{M}\pi(s(a))$. Since we already know that $s(\theta)\leq\pi(s(a))$, we conclude that
\[
s(\theta)=\pi(s(a)).
\]
Since $T(a)=h_\theta^{\frac1p}$ and $
s(h_\theta^{\frac1p})=s(h_\theta)=s(\theta)$, we finally obtain
\[
s(T(a))=s(\theta)=\pi(s(a)),
\]
which proves \eqref{support-S-pi}.
\end{proof}

The second property that we need is the bimodule structure of the range. This will eventually allow us to absorb the partial-isometry factor once we know that it belongs to $\pi(\cal{N})$.

\begin{prop}
\label{prop-bimodule}
Let $\cal{N}$ and $\cal{M}$ be $\sigma$-finite von Neumann algebras equipped with a normal faithful states. Suppose that $1 < p < \infty$ with $p \not=2$. Let $T \co \L^p(\cal{N}) \to \L^p(\cal{M})$ be a completely positive complete isometry. Then $T$ is a bimodule map. More precisely, we have \eqref{module-map} and
\begin{equation}
\label{S-right-module}
T(xy)
=\pi(x)T(y),\quad x \in \cal{N}, y \in \L^p(\cal{N}).
\end{equation}
In particular the subspace $\Ran T$ is a $\pi(\cal{N})$-bimodule.
\end{prop}

\begin{proof}
Since the map $T \co \L^p(\cal{N}) \to \L^p(\cal{M})$ is completely positive, it preserves adjoints. Hence, for any $x \in \cal{N}$ and any $y \in \L^p(\cal{N})$, we have
\[
T(xy)
=T((y^*x^*)^*)
=T(y^*x^*)^*
\ov{\eqref{module-map}}{=} (T(y^*)\pi(x^*))^*
=\pi(x^*)^*T(y^*)^*
=\pi(x)T(y).
\]
\end{proof}

Let $Y$ be a subspace of $\L^p(\cal{M})$. We define its left and right support projections by
\begin{equation}
\label{support-subspace}
s_\ell(Y)\ov{\mathrm{def}}{=}\bigvee_{y\in Y}s_\ell(y)
\quad\text{and}\quad
s_r(Y)\ov{\mathrm{def}}{=}\bigvee_{y\in Y}s_r(y).
\end{equation}
In particular, if $P \co \L^p(\cal{M}) \to \L^p(\cal{M})$ is a bounded projection, we write $s_\ell(P)\ov{\mathrm{def}}{=} s_\ell(\Ran P)$ and $s_r(P) \ov{\mathrm{def}}{=}s_r(\Ran P)$. Now, suppose that the projection $P \co \L^p(\cal{M}) \to \L^p(\cal{M})$ is positive. Following \cite[Section~6]{ArR24}, we define the support projection of $P$ by
\begin{equation}
\label{support-positive-projection}
s(P)
\ov{\mathrm{def}}{=}\bigvee_{\substack{h\in\Ran P\\ h\geq0}}s(h).
\end{equation}
Then by \cite[Section~6]{ArR24} we have
\begin{equation}
\label{supports-positive-projection}
s_\ell(P)=s_r(P)=s(P).
\end{equation}

We now turn to the main argument. Let $P \co \L^p(\cal M)\to\L^p(\cal M)$ be a positive contractive projection with $\Ran P\not=\{0\}$. Suppose that there exists a complete isometry $T \co \L^p(\cal{N}) \to \L^p(\cal{M})$ with range $\Ran P$. Since $T$ is a complete isometry, it is in particular a $2$-isometry. By the classification theorem of Junge, Ruan and Sherman \cite[Theorem~2 p.~287]{JRS05} and \cite[Remark 2 p.~310]{JRS05}, there exist a normal injective $*$-homomorphism $\pi \co \cal{N} \to \cal{M}$, a partial isometry $w \in \cal{M}$ and a normal conditional expectation $\E \co \cal{M} \to \cal{M}$ onto $\pi(\cal{N})$ such that
\begin{equation}
\label{JRS-factorization-here}
e
\ov{\mathrm{def}}{=}\pi(1)=w^*w
\end{equation}
and such that the map
\begin{equation}
\label{def-T-tilde}
\tilde{T}
\ov{\mathrm{def}}{=}w^*T \co \L^p(\cal{N}) \to \L^p(\cal{M})
\end{equation}
is a completely positive complete isometry. Moreover, in the construction of \cite[Theorem~4.9]{JRS05}, the restriction $F=\E|_{e\cal M e}\co e\cal{M} e \to \pi(\cal{N})$ is a faithful normal conditional expectation. By \cite[Theorem~2 p.~287]{JRS05}, we have 
\begin{equation}
\label{factorization-T}
T
=w\tilde{T}.
\end{equation}
Since $\Ran T=\Ran P$, it follows that
\begin{equation}
\label{Y-wZ}
\Ran P
=w\Ran\tilde T.
\end{equation}

\begin{lemma}
\label{sigma-finite}
The von Neumann algebra $\cal{N}$ is $\sigma$-finite.
\end{lemma}

\begin{proof} 
Since $\cal{M}$ is $\sigma$-finite, we can consider a normal faithful state $\varphi$ on $\cal{M}$. Since $\Ran P\not=\{0\}$, we have $\cal{N} \not=\{0\}$ and hence $e=\pi(1)\not=0$. Define
\[
\rho(x)
\ov{\mathrm{def}}{=}\frac{\varphi(\pi(x))}{\varphi(e)},\quad x\in\cal{N}.
\]
Then $\rho$ is a normal state on $\cal{N}$. Moreover, if $x\in\cal{N}_+$ and $\rho(x)=0$, then $\varphi(\pi(x))=0$. The faithfulness of $\varphi$ gives $\pi(x)=0$, and the injectivity of $\pi$ gives $x=0$. Thus $\rho$ is faithful. Consequently, $\cal{N}$ is $\sigma$-finite. 
\end{proof}

The first step is to compare the initial and final projections of $w$. Since $P$ is positive, its range is selfadjoint and therefore has identical left and right support projections. We now combine this observation with the support-preserving property of $\tilde T$. Since the von Neumann algebra $\cal{N}$ is $\sigma$-finite, choose with \cite[Exercise 7.6.46 p.~500]{KaR97b}, a faithful normal state $\rho$ on $\cal{N}$ and consider the positive element $a_0 \ov{\mathrm{def}}{=} h_\rho^{\frac{1}{p}}$ belonging to the space $\L^p(\cal{N})$. Then $s(a_0)=1$, and we obtain
\begin{equation}
\label{inter-5001}
s(\tilde{T}(a_0))
\ov{\eqref{support-S-pi}}{=} 
\pi(s(a_0))
=\pi(1)
\ov{\eqref{JRS-factorization-here}}{=} e.
\end{equation}
By Proposition~\ref{prop-bimodule}, for every $z \in \Ran \tilde{T}$ we have $ez=ze=z$. Hence $s_\ell(z) \leq e$ and $s_r(z) \leq e$ for any $z \in \Ran\tilde T$. On the other hand, $s(\tilde T(a_0)) \ov{\eqref{inter-5001}}{=} e$. Consequently, we have
\[
s_\ell(\Ran\tilde T)
=s_r(\Ran\tilde T)=e.
\]
For any $z \in \Ran \tilde{T}$, using $ez=ze=z$, we have
\[
(wz)^*(wz)
=z^*w^*wz
\ov{\eqref{JRS-factorization-here}}{=} z^*ez
=z^*z
\]
and
\[
(wz)(wz)^*
=wzz^*w^*.
\]
For any $z\in\Ran\tilde T$, we deduce with \eqref{support-projections} that
\[
s_r(wz)=s_r(z).
\]
Moreover, since $s_\ell(z) \leq e=w^*w$, we have
\[
s_\ell(wz)
=w s_\ell(z)w^*.
\]
Taking suprema and using $s_\ell(\Ran\tilde T)=s_r(\Ran\tilde T)=e$, we obtain using \eqref{Y-wZ}
\begin{equation}
\label{supports-Y-w}
s_r(\Ran P)
=e
\quad \text{and} \quad
s_\ell(\Ran P)
=ww^*.
\end{equation}
On the other hand, the contractive projection $P \co \L^p(\cal{M}) \to \L^p(\cal{M})$ is positive, hence it preserves adjoints. Thus $\Ran P$ is invariant under the involution. It follows that its left and right support projections coincide. Consequently, \eqref{supports-Y-w} yields
\begin{equation}
\label{w-unitary-corner}
ww^*
= e
\ov{\eqref{JRS-factorization-here}}{=} w^*w.
\end{equation}
Thus $w$ is a unitary element of the reduced von Neumann algebra $e\cal{M}e$. Now, we use the $\sigma$-finiteness assumption in an essential way. Since $P \co \L^p(\cal{M}) \to \L^p(\cal{M})$ is a positive contractive projection and since the von Neumann algebra $\cal{M}$ is $\sigma$-finite, the support result \cite[Proposition 6.1]{ArR24} provides a positive element $h \in \Ran P$ such that
\begin{equation}
\label{full-support-h}
s(h)
=s(P).
\end{equation}
For a positive projection $P$, its support $s(P)$ coincides with the left and right support projections of its range. Hence, by \eqref{supports-Y-w},
\begin{equation}
\label{support-h-e}
s(h)
=e.
\end{equation}
Consider the element $\xi \ov{\mathrm{def}}{=} T^{-1}(h)$ in the space $\L^p(\cal{N})$ and write its polar decomposition in the form
\begin{equation}
\label{polar-decompo}
\xi
=|\xi^*|v,
\end{equation}
where $v \in \cal N$ is a partial isometry. 

At this point we only know that $w$ is a unitary in the corner $e\cal{M}e$. This is not yet sufficient to identify $\Ran T$ with $\Ran\tilde T$. Indeed, for this purpose we need to prove that $w$ belongs to $\pi(\cal{N})$, since $\Ran\tilde T$ is a $\pi(\cal{N})$-bimodule. The next lemma establishes precisely this stronger conclusion. Its proof uses a positive element of $\Ran P$ with full support and the uniqueness of its polar decomposition.

\begin{lemma}
\label{lemma-belongs}
The element $w$ belongs to the von Neumann algebra $\pi(\cal{N})$. 
\end{lemma}

\begin{proof}
We first note by Proposition \ref{prop-bimodule} that the subspace $\Ran \tilde{T}$ is a $\pi(\cal{N})$-bimodule. Using the right module identity \eqref{S-right-module} in the last equality, we obtain
\begin{equation}
\label{polar-h-first}
h
=T(\xi)
\ov{\eqref{factorization-T}}{=} w \tilde{T}(\xi)
\ov{\eqref{polar-decompo}}{=} w \tilde{T}(|\xi^*|v)
\ov{\eqref{module-map}}{=} w \tilde{T}(|\xi^*|)\pi(v).
\end{equation}
Consider the positive element $
a
\ov{\mathrm{def}}{=} \tilde{T}(|\xi^*|)$ in the space $\L^p(\cal{M})$. Since the support of the element $|\xi^*|$ is $vv^*$, Proposition \ref{prop-supports} gives
\begin{equation}
\label{support-a}
s(a)
=s\big(\tilde{T}(|\xi^*|)\big)
\ov{\eqref{support-S-pi}}{=} \pi(s(|\xi^*|))
=\pi(vv^*).
\end{equation}
Define
\begin{equation}
\label{def-de-b}
b
\ov{\mathrm{def}}{=} \pi(v^*)a\pi(v).
\end{equation}
Then $b$ is positive. We claim that
\begin{equation}
\label{support-b}
s(b)=\pi(v^*v).
\end{equation}
Indeed, put $U \ov{\mathrm{def}}{=}\pi(v)$. By \eqref{support-a}, we have $s(a)=UU^*$. Moreover,
\[
b=U^*aU=(a^{\frac12}U)^*(a^{\frac12}U).
\]
Hence $s(b)$ is the orthogonal projection onto the complement of $\ker(a^{\frac12}U)$. Since $s(a)=UU^*$, the operator $a^{\frac12}$ is injective on $UU^*H$. Since $U\eta\in UU^*H$ for every $\eta$, we consequently have
\[
a^{\frac12}U\eta=0
\quad\Longleftrightarrow\quad
U\eta=0
\quad\Longleftrightarrow\quad
U^*U\eta=0.
\]
It follows that $s(b)=U^*U=\pi(v^*v)$
%
, proving \eqref{support-b}. Consequently, we have
\begin{equation}
\label{polar-h-second}
h
\ov{\eqref{polar-h-first}}{=} w a \pi(v)
\ov{\eqref{support-a}}{=} w\pi(vv^*)a\pi(v)
=w\pi(v)\pi(v^*)a\pi(v)
\ov{\eqref{def-de-b}}{=} 
w\pi(v)b.
\end{equation}
Put $u \ov{\mathrm{def}}{=} w\pi(v)$. We have
\begin{equation}
\label{polar-2}
h
=ub.
\end{equation}
Using \eqref{w-unitary-corner}, we see that
\begin{equation}
\label{39300}
u^*u 
=\pi(v^*)w^*w\pi(v)
\ov{\eqref{w-unitary-corner}}{=} \pi(v^*)e\pi(v)
\ov{\eqref{JRS-factorization-here}}{=} \pi(v^*)\pi(1)\pi(v)
= \pi(v^*v)
\ov{\eqref{support-b}}{=} s(b).
\end{equation}
So $u$ is a partial isometry whose initial projection is the support of $b$. It follows from \eqref{polar-h-second} and $u = w\pi(v)$ that
\[
h^*h
\ov{\eqref{polar-2}}{=} (ub)^*ub
=bu^*ub
\ov{\eqref{39300}}{=} bs(b)b
=b^2.
\]
Hence $|h|=b$, and \eqref{polar-2} is the polar decomposition of $h$. But the element $h$ is positive and $s(h) \ov{\eqref{support-h-e}}{=} e$. Therefore the partial isometry of its polar decomposition is $e$, and we obtain
\begin{equation}
\label{wpi-v-e}
w\pi(v)
=e.
\end{equation}
Now, we have $\pi(v^*v) \ov{\eqref{support-b}}{=} s(b) = s(|h|) = e$. Since the map $\pi \co \cal{N} \to \cal{M}$ is injective and by \eqref{JRS-factorization-here}, we obtain $v^*v=1$. On the other hand, by \eqref{support-projections}, the left support projection of $w\pi(v)$ is $w\pi(v)(w\pi(v))^*=w\pi(vv^*)w^*$ and the left projection of $e$ is of course $e$. Taking left support projections in \eqref{wpi-v-e}, we obtain
\[
e
=w\pi(vv^*)w^*.
\]
Multiplying on the left by $w^*$ and on the right by $w$ gives $w^*ew=w^*w\pi(vv^*)w^*w$. Using \eqref{w-unitary-corner}, we obtain $e=e\pi(vv^*)e$. Since $\pi(vv^*)\leq \pi(1)=e$, we infer that $e=\pi(vv^*)$. Hence, by injectivity, $vv^*=1$. Thus $v$ is unitary in $\cal N$, and \eqref{wpi-v-e} yields that the element $w=\pi(v^*)$ belongs to $\pi(\cal{N})$.
\end{proof}

We can now prove the main result of this section.

\begin{thm}
\label{thm-positive-projection-ci-range}
Let $\cal M$ be a $\sigma$-finite von Neumann algebra and suppose that $1<p<\infty$ with $p\not=2$. Let $P\co\L^p(\cal M)\to\L^p(\cal M)$ be a positive contractive projection. The following assertions are equivalent.
\begin{enumerate}
	\item $P$ is completely positive.
	\item $\Ran P$ is completely isometric to a noncommutative $\L^p$-space.
\end{enumerate}
\end{thm}

\begin{proof}
If $P=0$, the conclusion is immediate. We henceforth assume that $P \ne 0$.

1. $\Rightarrow$ 2.: Suppose first that $P \co \L^p(\cal{M}) \to \L^p(\cal{M})$ is completely positive. Then $P$ is in particular $2$-positive. By \cite[Theorem~1.1]{ArR24}, the range $\Ran P$ is completely order and completely isometrically isomorphic to a noncommutative $\L^p$-space $\L^p(\cal{N})$ and the von Neumann algebra $\cal N$ arising in this theorem is also $\sigma$-finite. This proves the first implication.

2. $\Rightarrow$ 1.: If $\Ran P=\{0\}$, then $P=0$ and the conclusion is immediate. Suppose therefore that $\Ran P\not=\{0\}$. By the second point there exist a von Neumann algebra $\cal{N}$ and a surjective complete isometry $T\co\L^p(\cal{N})\to\Ran P$. We use the notation introduced before \eqref{Y-wZ}. As shown in Lemma \ref{sigma-finite}, $\cal{N}$ is  $\sigma$-finite. Note that the subspace $\Ran \tilde{T}$ is a $\pi(\cal{N})$-bimodule by Proposition \ref{prop-bimodule} and that $w$ belongs to the von Neumann algebra $\pi(\cal{N})$ according to Lemma \ref{lemma-belongs}. We deduce that
\begin{equation}
\label{inter-3500}
w\Ran \tilde{T}
=\Ran \tilde{T}.
\end{equation}
Combining this with \eqref{Y-wZ}, we conclude that
\begin{equation}
\label{Y-equals-Z}
\Ran P
\ov{\eqref{Y-wZ}}{=} w\Ran \tilde{T}
\ov{\eqref{inter-3500}}{=}\Ran \tilde{T}.
\end{equation}
Since $\tilde{T} \co \L^p(\cal{N}) \to \L^p(\cal{M})$ is a positive complete isometry, \cite[Theorem 4.9 p.~308 and Remark 2 p.~310]{JRS05} provides a completely positive contractive projection $Q \co \L^p(\cal{M}) \to \L^p(\cal{M})$ such that $\Ran Q = \Ran \tilde{T}$. By \eqref{Y-equals-Z}, we conclude that
\[
\Ran Q
=\Ran P
\ov{\eqref{Y-equals-Z}}{=} \Ran \tilde{T}.
\]
Recall that by \cite[Corollary 5.2 p.~1480]{PiX03} the Banach space $\L^p(\cal{M})$ is smooth for $1 < p < \infty$. It remains only to identify $P$ and $Q$ with Proposition \ref{prop-Sullivan}. Hence $P=Q$. Since the map $Q$ is completely positive, so is $P$.
\end{proof}

Combining with the results of \cite{ArR24} and \cite{JRS05}, we deduce immediately the next result.

\begin{cor}
\label{cor-2-isometry-characterization}
Let $\cal{M}$ be a $\sigma$-finite von Neumann algebra and suppose that $1<p<\infty$ with $p\not=2$. Let $P\co\L^p(\cal{M})\to\L^p(\cal{M})$ be a positive contractive projection. The following assertions are equivalent.
\begin{enumerate}
\item The projection $P$ is completely positive.
\item The projection $P$ is $2$-positive.
\item The range $\Ran P$ is completely isometric to a noncommutative $\L^p$-space.
\item There exist a von Neumann algebra $\cal{N}$ and a surjective $2$-isometry $T\co\L^p(\cal{N})\to\Ran P$.
\item The range $\Ran P$, equipped with the order inherited from $\L^p(\cal{M})$, is completely order and completely isometrically isomorphic to a noncommutative $\L^p$-space.
\end{enumerate}
\end{cor}

The proof of Theorem \ref{thm-positive-projection-ci-range} equally shows the following result.

\begin{cor}
\label{cor-unitary-rectification}
Let $\cal{M}$ be a $\sigma$-finite von Neumann algebra and suppose that $1<p<\infty$ with $p\not=2$. Let $P\co\L^p(\cal{M})\to\L^p(\cal{M})$ be a positive contractive projection and suppose that $
T\co\L^p(\cal{N})\to \Ran P$ is a surjective complete isometry. Then there exists a unitary $v\in\cal{N}$ such that the map $S \co \L^p(\cal{N}) \to \Ran P$, $x \mapsto T(vx)$, is a completely positive complete isometry. In particular, every complete isometry onto $\Ran P$ differs from a completely positive complete isometry by left multiplication by a unitary of the source algebra.
\end{cor}

\begin{proof}
We may assume that $\Ran P\not=\{0\}$. Apply the Junge--Ruan--Sherman factorization to the complete isometry $T$. With the notation used in the proof of Theorem~\ref{thm-positive-projection-ci-range}, there exist an injective normal $*$-homomorphism $\pi\co\cal{N}\to\cal{M}$, a partial isometry $w\in\cal{M}$, and a completely positive complete isometry $\tilde T\co\L^p(\cal{N})\to\L^p(\cal{M})$ such that $T=w\tilde{T}$ with $w^*w=\pi(1)$. The proof of Lemma~\ref{lemma-belongs} shows more precisely that there exists a unitary $v\in\cal{N}$ such that
\begin{equation}
\label{unitary-rectification-relation}
w\pi(v)=\pi(1).
\end{equation}
By Proposition~\ref{prop-bimodule}, the map $\tilde T$ is a left $\pi(\cal{N})$-module map. Hence, for every $x\in\L^p(\cal{N})$,
\[
T(vx)
=w\tilde T(vx)
\ov{\eqref{S-right-module}}{=} w \pi(v)\tilde T(x)
\ov{\eqref{unitary-rectification-relation}}{=}
\pi(1)\tilde T(x)
=\tilde T(x).
\]
Thus the map $S\co\L^p(\cal{N})\to\Ran P$, $x\mapsto T(vx)$, coincides with $\tilde T$. Consequently, $S$ is a completely positive complete isometry.
\end{proof}

We conclude this section with two examples showing that both the positivity of the projection and the complete nature of the isometry assumption are essential.

\begin{remark} \normalfont
\label{rem-positivity-essential}
The positivity assumption on the projection in Theorem~\ref{thm-positive-projection-ci-range} cannot be removed. Indeed, consider the von Neumann algebra $\cal M=\M_2$ and the map $P\co S_2^p \to S_2^p$ defined by 
$$
P(x)
\ov{\mathrm{def}}{=} e_{11} x e_{22}.
$$
Then $P^2=P$ and $P$ is completely contractive, since it is the composition of left and right multiplication by projections. Its range is $
\Ran P
=\mathbb{C} e_{12}$, which is completely isometric to $\L^p(\mathbb{C})=\mathbb C$. More precisely, for any integer $n\geq1$, the map $S_n^p \to \L^p(\M_n\otvn\M_2)$,  $a \mapsto a\ot e_{12}$, is an isometry. However, $P$ is not positive. For instance, if we consider the positive element
$
x=
\begin{bmatrix}
1&1\\
1&1
\end{bmatrix}$, we have $P(x)=e_{12}$ and $e_{12}$ is not positive.

Thus, the abstract complete isometry class of the range does not determine the order structure inherited from the ambient noncommutative $\L^p$-space. In contrast, if one assumes that there exists a positive complete isometry $T\co\L^p(\cal N)\to\L^p(\cal M)$ with $\Ran T = \Ran P$, then the positivity of $P$ is no longer needed. Indeed, the Junge--Ruan--Sherman theorem yields a completely positive contractive projection $Q$ onto $\Ran T$. Since $\Ran Q=\Ran P$ and $\L^p(\cal M)$ is smooth for $1 < p < \infty$, Proposition~\ref{prop-Sullivan} implies $P=Q$. Hence $P$ is completely positive.
\end{remark}

\begin{remark} \normalfont
\label{rem-complete-isometry-essential}
The complete isometry assumption in Theorem~\ref{thm-positive-projection-ci-range} cannot be replaced by an isometry, even if the latter preserves the order structure. Indeed, let $n\geq2$, let $\cal{M}=\M_n\oplus\mathord{\M_n}$ and define $\alpha\co\L^p(\cal{M})\to\L^p(\cal{M})$ by 
$$
\alpha(x,y)
\ov{\mathrm{def}}{=} (y^{\top},x^{\top}).
$$
Since the transpose is a positive isometry on $S_n^p$, the map $\alpha$ is a positive isometric involution. Consequently, $
P\ov{\mathrm{def}}{=}\frac{1}{2}(\Id+\alpha)$ is a positive contractive projection and we have
\[
\Ran P
=\{(x,x^{\top}):x\in S_n^p\}.
\]
Moreover, the map $T \co S_n^p \to \Ran P$ defined by
\[
T(x)
\ov{\mathrm{def}}{=} 2^{-\frac1p}(x,x^{\top}),
\]
is a surjective positive isometry whose inverse is also positive. Thus $\Ran P$ is order isometrically isomorphic to the noncommutative $\L^p$-space $S_n^p$.

On the other hand, $P$ is not completely positive. Indeed, consider the positive element $q=\sum_{i,j=1}^n e_{ij}\otimes e_{ij}$ in the algebra $\M_n \ot \M_n$. Then $(0,q)$ is positive in $\M_n\otimes\cal{M}$, whereas the first component of $(\Id_{\M_n}\otimes P)(0,q)$ is $\frac12\sum_{i,j=1}^n e_{ij} \ot e_{ji}$, which is not positive. Hence $P$ is not completely positive. This shows that neither the Banach space structure nor the order structure of $\Ran P$ at the scalar level is sufficient in Theorem~\ref{thm-positive-projection-ci-range}. The complete isometry assumption is essential.
\end{remark}

\section{Subspaces completely isometric to a rectangular $\L^p$-space}
\label{sec-rectangular-spaces}

The aim of this section is to prove a rectangular analogue of the result of the previous section. The main difference is that a rectangular noncommutative $\L^p$-space is not itself an ordinary noncommutative $\L^p$-space. We overcome this difficulty by stabilization: after amplification by $S^p$, the rectangular space becomes completely isometric to an ordinary noncommutative $\L^p$-space, to which the Junge--Ruan--Sherman theorem applies.

Following \cite{Arh24b} (see also \cite[p.~869]{KaR02}), we define the rectangular $\L^p$-spaces of $\W^*$-$\TRO$s using Kosaki noncommutative $\L^p$-spaces of \cite{Kos84} \cite{Ray03}, relying on the theory of complex interpolation \cite{BeL76}. Let $V$ be a $\W^*$-$\TRO$ with linking von Neumann algebra $\mathrm{R}(V)$. Suppose  that the von Neumann algebra $\mathrm{R}(V)$ is $\sigma$-finite equipped with a normal faithful state $\varphi$. Suppose that $1 \leq p < \infty$. We define the rectangular $\L^p$-space $\L^p(V,\varphi)$ to be the norm closure of $V=e \mathrm{R}(V)e^\perp$ in the Kosaki noncommutative $\L^p$-space $\L^p(\mathrm{R}(V),\varphi)$. Sometimes, we use the notation $\L^p(V)$. It is easy to check that $\L^p(V,\varphi)=e\L^p(R(V),\varphi)e^\perp$ since $R(V)$ is dense in the Banach space $\L^p(R(V),\varphi)$. Thus the rectangular $\L^p$-space associated with $V$ is exactly the $\L^p$-analogue of the off-diagonal corner realizing $V$ inside its linking von Neumann algebra. We let $\L^\infty(V,\varphi) \ov{\mathrm{def}}{=} V$.


Decomposable maps are a generalization of completely positive maps. Recall that a linear map $T \co \L^p(\cal{M}) \to \L^p(\cal{N})$ between noncommutative $\L^p$-spaces associated to von Neumann algebras $\cal{M}$ and $\cal{N}$ is decomposable \cite[(3.2)]{JuR04} if there exist bounded linear maps $v_1,v_2 \co \L^p(\cal{M}) \to \L^p(\cal{N})$ such that the linear map
\begin{equation}
\label{Matrice-2-2-Phi-intro}
\Phi
\ov{\mathrm{def}}{=}\begin{bmatrix}
   v_1  &  T \\
   T^\circ  &  v_2  \\
\end{bmatrix}
\co S^p_2(\L^p(\cal{M})) \to S^p_2(\L^p(\cal{N})), \quad \begin{bmatrix}
   a  &  b \\
   c &  d  \\
\end{bmatrix}\mapsto 
\begin{bmatrix}
   v_1(a)  &  T(b) \\
   T^\circ(c)  &  v_2(d)  \\
\end{bmatrix}
\end{equation}
is completely positive, where $T^\circ(c) \ov{\mathrm{def}}{=} T(c^*)^*$.  In this case, we let
\begin{equation}
\label{Norm-dec-intro}
\norm{T}_{\dec,\L^p(\cal{M}) \to \L^p(\cal{N})}
\ov{\mathrm{def}}{=} \inf \max\{\norm{v_1},\norm{v_2}\}
\end{equation}
where the infimum is taken over all maps $v_1$ and $v_2$. We say that $T$ is contractively decomposable if $\norm{T}_{\dec,\L^p(\cal{M}) \to \L^p(\cal{N})} \leq 1$. Note that if the von Neumann algebras $\cal{M}$ and $\cal{N}$ are hyperfinite, it is equivalent to saying that $T$ is contractively regular by \cite[Theorem 3.24]{ArK23}, which means that for any operator space $E$, the map $T \ot \Id_E$ induces a contraction between the vector-valued noncommutative $\L^p$-spaces $\L^p(\cal{M},E)$ and $\L^p(\cal{N},E)$.

By \cite[Theorem~1.3]{Bof26}, the range of a contractively decomposable projection is completely isometrically isomorphic to a corner $e\L^p(\cal N)e^\perp$. Conversely, \cite[Proposition~5.2]{Arh24b} shows that a canonical corner $e\L^p(\cal{N})e^\perp$ is contractively decomposably complemented in $\L^p(\cal{N})$. The point below is that the same conclusion remains true for an arbitrary completely isometric copy of such a corner inside another noncommutative $\L^p$-space. In other words, contractive decomposable complementability depends only on the complete isometry class of the rectangular $\L^p$-space and not on its particular realization as a corner.

The proof relies on a stabilization argument. After amplification by $S^p$, a rectangular $\L^p$-space becomes completely isometric to an ordinary noncommutative $\L^p$-space. We may then apply the classification of $2$-isometries of Junge, Ruan and Sherman. Their construction yields a contractively decomposable projection onto the stabilized copy. Compressing this projection to a matrix corner gives a contractively decomposable projection onto the original copy. Notice that the complete isometry assumption is essential in this argument, since it is precisely what allows us to pass to the $S^p$-amplification.

\begin{thm}
\label{thm-characterization-range-decomposable-rectangular}
Let $\cal M$ be a $\sigma$-finite von Neumann algebra with separable predual equipped with a normal faithful state and suppose that $1 < p  <\infty$ with $p \not= 2$. Let $Y \not=\{0\}$ be a closed subspace of $\L^p(\cal{M})$. The following assertions are equivalent.
\begin{enumerate}
\item The subspace $Y$ is the range of a contractively decomposable projection on the Banach space $\L^p(\cal M)$.

\item There exist a $\sigma$-finite von Neumann algebra $\cal N$ equipped with a normal faithful positive linear form $\psi$, a projection $e \in \cal{N}$, and a surjective complete isometry $T \co e\L^p(\cal{N},\psi)e^\perp \to Y$.

\item There exist a $\W^*$-$\TRO$ $V$ whose linking von Neumann algebra $\mathrm{R}(V)$ is $\sigma$-finite, a normal faithful state $\varphi$ on $\mathrm{R}(V)$, and a surjective complete isometry $
T \co \L^p(V,\varphi) \to Y$.

\end{enumerate}
\end{thm}

\begin{proof}
1. $\Rightarrow$ 2.: This implication follows from \cite[Theorem~1.3]{Bof26}. 

2. $\Rightarrow$ 1.: Let $Y$ be a closed subspace of the Banach space $\L^p(\cal{M})$ and suppose that there exists a surjective complete isometry $T \co e\L^p(\cal{N})e^\perp \to Y$. Since noncommutative $\L^p$-spaces are independent, up to complete order isometry, of the choice of the reference weight, we suppress the weights in the rest of the proof. The proof proceeds in three steps. We first stabilize the rectangular source and identify it with an ordinary noncommutative $\L^p$-space. We then construct a contractively decomposable projection onto the stabilized copy $S^p(Y)$. Finally, we compress this projection to the $(1,1)$-matrix corner in order to obtain a projection onto $Y$.

We first remove the part of $\cal{N}$ which is invisible to the rectangular corner. This reduction ensures that both corner projections have full central support, a property which will be needed after stabilization to make them equivalent to the unit. We first reduce to the case where $e$ and $e^\perp$ have full central support. Let $c(e)$ and $c(e^\perp)$ denote their central supports in $\cal N$ and set $z \ov{\mathrm{def}}{=} c(e)c(e^\perp)$. Any $x \in e\L^p(\cal{N})e^\perp$ satisfies $x=zx$, and hence
\[
e\L^p(\cal N)e^\perp
=ze\L^p(z\cal{N})ze^\perp.
\]
Replacing $\cal{N}$ by $z\cal{N}$, we may consequently assume that 
\begin{equation}
\label{support-central}
c(e)=c(e^\perp)=1
\end{equation}

Now, we stabilize by tensoring with $\B(\ell^2)$. The purpose of this amplification is to turn the two full corner projections into properly infinite projections. They will therefore become Murray--von Neumann equivalent, allowing us to convert the rectangular corner into a square one.

\begin{lemma}
The projections $f \ov{\mathrm{def}}{=} 1 \ot e$ and $f^\perp$ have central support $1$ in the von Neumann algebra $\B(\ell^2) \otvn \cal{N}$, and both are properly infinite. 
\end{lemma}

\begin{proof}
We have
\[
c_{\B(\ell^2) \otvn \cal N}(f)
=1\ot c_{\cal N}(e)
=1_{\B(\ell^2) \otvn \cal N}
\quad\text{and}\quad
c_{\B(\ell^2) \otvn \cal N}(f^\perp)
=1\ot c_{\cal N}(e^\perp)=1_{\B(\ell^2) \otvn \cal N}.
\]
We next note that $f$ and $f^\perp$ are properly infinite. Let $s_1,s_2 \in \B(\ell^2)$ be two isometries with orthogonal ranges, so that $s_i^*s_i=1$ and $s_1s_1^*\perp s_2s_2^*$. For $i=1,2$, put $v_i=s_i \ot e$. Then
\[
v_i^*v_i
=1 \ot e=f,
\qquad 
v_iv_i^*
=s_is_i^*\ot e
\leq f
\]
and $v_1v_1^*\perp v_2v_2^*$. Thus $f$ contains two orthogonal subprojections, each Murray--von Neumann equivalent to $f$, and hence $f$ is properly infinite. Replacing $e$ by $e^\perp$ gives the same conclusion for $1-f$.
\end{proof}

Since the von Neumann algebra $\B(\ell^2) \otvn \cal{N}$ has a separable predual, by Proposition \ref{prop-Blackadar}, we have $f \sim 1_{\B(\ell^2) \otvn \cal{N}} \sim f^\perp$. Choose a partial isometry $u \in \B(\ell^2) \otvn \cal{N}$ such that
\begin{equation}
\label{partial-iso-39}
uu^*
=f
\quad \text{and} \quad 
u^*u
=f^\perp.
\end{equation}
Multiplication by $u$ therefore identifies the square corner supported by $f$ with the off-diagonal corner from $f$ to $f^\perp$. This is the basic mechanism which turns the stabilized rectangular $\L^p$-space into an ordinary noncommutative $\L^p$-space.

\begin{lemma}
The map
\[
\Theta \co \L^p(f(\B(\ell^2) \otvn \cal N) f) \to f\L^p(\B(\ell^2) \otvn \cal{N})f^\perp,\quad a \mapsto au,
\]
is a surjective complete isometry. 
\end{lemma}

\begin{proof}
Note the identifications $\L^p(f(\B(\ell^2) \otvn \cal N) f) = f\L^p(\B(\ell^2) \otvn \cal N)f $ and $\L^p(f(\B(\ell^2) \otvn \cal{N})f^\perp)=f\L^p(\B(\ell^2) \otvn \cal{N})f^\perp$. The multiplication operator $\Theta $is clearly contractive and even completely contractive. Since $xu^*u \ov{\eqref{partial-iso-39}}{=} xf^\perp=x$ for any $x \in f\L^p(\B(\ell^2) \otvn \cal{N})f^\perp$ and $yuu^* \ov{\eqref{partial-iso-39}}{=} yf=y$ for any $y \in f\L^p(\B(\ell^2) \otvn \cal N)f$, its inverse is $\Theta^{-1} \co f\L^p(\B(\ell^2) \otvn \cal{N})f^\perp \to \L^p(f(\B(\ell^2) \otvn \cal N) f)$.  $a \mapsto au^*$. This map is also completely contractive. The conclusion is obvious, see also \cite[(1.2.7)]{BLM04}.
\end{proof}

We next identify this off-diagonal corner with the $S^p$-amplification of the original rectangular space. By the Fubini identification for vector-valued Schatten spaces \cite[(3.6) p.~40]{Pis98}, we have $
S^p(\L^p(\cal N))
=\L^p(\B(\ell^2) \otvn \cal{N})$ completely isometrically. Taking the corresponding corners gives
\begin{equation}
\label{rectangular-stabilization}
S^p(e\L^p(\cal{N})e^\perp)
=f\L^p(\B(\ell^2) \otvn \cal{N}))f^\perp.
\end{equation}
Combining this identification with the preceding lemma, we have therefore transformed the stabilized rectangular space into an ordinary noncommutative $\L^p$-space.

Since $T \co e\L^p(\cal{N})e^\perp \to Y$ is a complete isometry, by \cite[Corollary 1.2 p.~19]{Pis98} its $S^p$-amplification
\[
\Id_{S^p} \ot T \co S^p(e\L^p(\cal{N}) e^\perp) \to S^p(Y)
\]
is a complete isometry. Hence by composition
\[
\widetilde{T}
\ov{\mathrm{def}}{=}(\Id_{S^p}\ot T)\Theta \co \L^p(f(\B(\ell^2) \otvn \cal{N}) f) \to S^p(Y)
\]
is a surjective complete isometry. Using again the Fubini identification, we regard the space $S^p(Y)$ as a subspace of the space $
S^p(\L^p(\cal M))
=\L^p(\B(\ell^2) \otvn \cal{M})$. Thus $\widetilde{T}$ is a complete isometry from the noncommutative $\L^p$-space $\L^p(f(\B(\ell^2) \otvn \cal{N}) f)$ into the space $\L^p(\B(\ell^2) \otvn \cal{M})$, with
\begin{equation}
\label{range-stabilized-T}
\Ran\widetilde T
=S^p(Y).
\end{equation}
Now, we are now in the setting of the Junge--Ruan--Sherman theorem: $\widetilde T$ is defined on an ordinary noncommutative $\L^p$-space and its range is precisely the stabilized subspace $S^p(Y)$. We may therefore apply \cite[Theorem~2 and Remark~2]{JRS05}. Since $p\not=2$, there exist a partial isometry $w \in \B(\ell^2) \otvn \cal{M}$ and a completely positive complete isometry
\[
V \co \L^p(f(\B(\ell^2) \otvn \cal{N}) f) \to \L^p(\B(\ell^2) \otvn \cal{M})
\]
such that $\widetilde{T}=\Mult_{w,1}V$, where $\Mult_{w,1}(x)=wx$. Moreover, there exists a completely positive contractive projection $R \co \L^p(\B(\ell^2) \otvn \cal{M}) \to \L^p(\B(\ell^2) \otvn \cal{M})$ whose range is the subspace $\Ran V$.  Put $e_0 \ov{\mathrm{def}}{=} w^*w$. We consider the multiplication operators $\Mult_{w,1}\co\L^p(\B(\ell^2) \otvn \cal{M})\to\L^p(\B(\ell^2) \otvn \cal{M})$, $x \mapsto wx$ and $\Mult_{w^*,1} \co \L^p(\B(\ell^2) \otvn \cal{M}) \to \L^p(\B(\ell^2) \otvn \cal{M})$, $x \mapsto w^*x$. By the Junge--Ruan--Sherman factorization, $e_0$ is the unit of the von Neumann algebra arising from the normal $*$-homomorphism associated with $V$. Moreover, the bimodule property of $V$ implies that
\begin{equation}
\label{support-range-V}
e_0z
=ze_0
=z,\quad z \in \Ran V.
\end{equation}
In particular, the restriction $\Mult_{w,1}\co \Ran V \to w\Ran V$ is a surjective isometry whose inverse is the restriction of $\Mult_{w^*,1}$ to $w\Ran V$. The projection $R$ has range $\Ran V$, whereas the subspace that we need to complement is $w\Ran V=\Ran\widetilde T$. We therefore transport $R$ through the partial isometry $w$. Define the linear map
\begin{equation}
\label{def-de-pi}
\Pi
\ov{\mathrm{def}}{=}
\Mult_{w,1}R\Mult_{w^*,1}
\co\L^p(\B(\ell^2) \otvn \cal{M})\to\L^p(\B(\ell^2) \otvn \cal{M}).
\end{equation}
We claim that the map $\Pi$ is a projection onto $w\Ran V$. Indeed, for any $x \in \L^p(\B(\ell^2) \otvn \cal{M})$, we have $R(w^*x)\in\Ran V$, and therefore
\[
\Pi(x)
=wR(w^*x) \in w\Ran V.
\]
Conversely, let $y \in w\Ran V$. Write $y=wz$ with $z\in\Ran V$. Using \eqref{support-range-V} and the fact that $R$ is a projection onto $\Ran V$, we obtain
\[
\Pi(y)
=\Pi(wz)
\ov{\eqref{def-de-pi}}{=} wR(w^*wz)
=wR(e_0z)
\ov{\eqref{support-range-V}}{=}wR(z)
=wz
=y.
\]
Consequently, the map $\Pi$ is the identity on the subspace $w\Ran V$ and has range contained in $w\Ran V$. Hence $\Pi^2=\Pi$ and
\begin{equation}
\label{range-Pi-stabilized}
\Ran\Pi
=w\Ran V
=\Ran\widetilde T
\ov{\eqref{range-stabilized-T}}{=}
S^p(Y).
\end{equation}
It remains to check that this projection $\Pi$ is contractively decomposable. This follows from its factorization into a completely positive contraction and left multiplication operators. Indeed, by \cite[Lemma 5.1]{Arh24b} the  operators $\Mult_{w,1} \co \L^p(\B(\ell^2) \otvn \cal{M}) \to \L^p(\B(\ell^2) \otvn \cal{M})$ and $\Mult_{w^*,1} \co \L^p(\B(\ell^2) \otvn \cal{M}) \to \L^p(\B(\ell^2) \otvn \cal{M})$ are decomposable with
\[
\norm{\Mult_{w,1}}_{\dec, \L^p(\B(\ell^2) \otvn \cal{M}) \to \L^p(\B(\ell^2) \otvn \cal{M})}
\leq 1
\quad \text{and} \quad 
\norm{\Mult_{w^*,1}}_{\dec,\L^p(\B(\ell^2) \otvn \cal{M}) \to \L^p(\B(\ell^2) \otvn \cal{M})}
\leq 1.
\]
Since $R$ is completely positive and contractive, it is contractively decomposable by mimicking the proof of \cite[Proposition 3.11 p.~30]{ArK23} (stated in the semifinite case). By the submultiplicativity of the decomposable norm \cite[(3.3)]{JuR04}, we obtain
\[
\norm{\Pi}_{\dec}
\ov{\eqref{def-de-pi}}{=} \norm{\Mult_{w,1}R \Mult_{w^*,1}}_{\dec}
\leq \norm{\Mult_{w,1}}_{\dec} \norm{R}_{\dec} \norm{\Mult_{w^*,1}}_{\dec}
\leq 1.
\]
We have constructed a contractively decomposable projection onto $S^p(Y)$. The last step is to recover $Y$ itself. Since $Y$ is canonically the $(1,1)$-matrix corner of $S^p(Y)$, it suffices to compress $\Pi$ to this corner. Consider the element $q \ov{\mathrm{def}}{=} e_{11} \ot 1_{\cal M}$ of the von Neumann algebra $\B(\ell^2) \otvn \cal{M}$. Under the canonical complete order isometry
\[
j \co \L^p(\cal M) \to q\L^p(\B(\ell^2) \otvn \cal{M})q,\quad x \mapsto e_{11} \ot x,
\]
define the map $Q \co \L^p(\cal{M}) \to \L^p(\cal{M})$ by
\begin{equation}
\label{compression-Pi}
j(Qx)
=q\Pi(j(x))q,\quad x \in \L^p(\cal M).
\end{equation}
The map $j$, its inverse on $q\L^p(\B(\ell^2) \otvn \cal{M})q$, and the compression $x \mapsto qxq$ are completely positive contractions. They are therefore contractively decomposable. Combining this fact with $\norm{\Pi}_{\dec}\leq1$ gives
\begin{equation}
\label{dec-Q}
\norm{Q}_{\dec,\L^p(\cal{M}) \to \L^p(\cal{M})}
\leq 1.
\end{equation}
We claim that the map $Q \co \L^p(\cal{M}) \to \L^p(\cal{M})$ is a projection onto the subspace $Y$. First, if $x \in \L^p(\cal{M})$, then by \eqref{range-Pi-stabilized},
\[
\Pi(j(x)) \in S^p(Y).
\]
The $(1,1)$-corner of $S^p(Y)$ is precisely $j(Y)$. Thus \eqref{compression-Pi} implies that $Q(x)$ belongs to $ Y$. Conversely, if $y \in Y$, then $j(y)=e_{11}\ot y$ belongs to $S^p(Y) = \Ran \Pi$. Hence $\Pi(j(y))=j(y)$ and therefore $Q(y)=y$. We conclude that $\Ran Q=Y$ and $Q^2=Q$.


It remains to prove the equivalence between the second point and the third point.

2. $\Rightarrow$ 3.: Let $V \ov{\mathrm{def}}{=} e\cal{N}e^\perp$. Set $
z
\ov{\mathrm{def}}{=}c(e)c(e^\perp)$, where the central supports are taken in the von Neumann algebra $\cal{N}$. Since $Y \not= \{0\}$, we have $V \not= \{0\}$ and hence $z\not=0$. Any $x \in e\L^p(\cal{N})e^\perp$ satisfies $x=zx$. So
\begin{equation}
\label{corner-central-reduction}
e\L^p(\cal{N})e^\perp
=ze\L^p(z\cal{N})ze^\perp
\end{equation}
completely isometrically. 
Observe that we have
\[
c(ze)
\ov{\eqref{support-central-et-z}}{=} zc(e)
\ov{\eqref{support-central}}{=}z
\quad\text{and}\quad
c(ze^\perp)
\ov{\eqref{support-central-et-z}}{=} zc(e^\perp)
\ov{\eqref{support-central}}{=}z.
\]
Thus $ze$ and $ze^\perp$ have full central support in the von Neumann algebra $z\cal{N}$. Moreover, we have
\[
V=e\cal{N}e^\perp=ze\cal{N}ze^\perp.
\]
Indeed, the inclusion from right to left is immediate. Conversely, if $x\in e\cal{N}e^\perp$, then $c(e)x=x$ and $c(e^\perp)x=x$, since $c(e^\perp)$ is central. Hence $zx=x$, and similarly $xz=x$. So$x=zxz$ belongs to $ze\cal{N}ze^\perp$.

Now, we identify the linking von Neumann algebra of $V$. Since the projection $ze^\perp$ has full central support in the von Neumann algebra $z\cal{N}$, we have
\[
\overline{\Span (z\cal{N})(ze^\perp)(z\cal{N})}^{\w^*}=z\cal{N}.
\]
Therefore,
\[
\mathrm{M}(V)
\ov{\eqref{def-MV-NV}}{=}
\overline{\Span VV^*}^{\w^*}
=ze\, \overline{\Span (z\cal{N})(ze^\perp)(z\cal{N})}^{\w^*}ze
=ze\cal{N}e.
\]
Similarly, since $ze$ has full central support in $z\cal{N}$, we have
\[
\mathrm{N}(V)
\ov{\eqref{def-MV-NV}}{=}
\overline{\Span V^*V}^{\w^*}
=ze^\perp\cal{N}e^\perp.
\]
Thus, relative to the decomposition $z=ze+ze^\perp$, we obtain
\[
\mathrm{R}(V)
\ov{\eqref{Linking-algebra}}{=}
\begin{bmatrix}
\mathrm{M}(V)   &  V \\
V^*   &  \mathrm{N}(V) \\
\end{bmatrix}
\cong
\begin{bmatrix}
ze\cal{N}e & e\cal{N}e^\perp\\
e^\perp\cal{N}e & ze^\perp\cal{N}e^\perp
\end{bmatrix}
\simeq z\cal{N}.
\]

Let $\psi_z$ denote the restriction of $\psi$ to the von Neumann algebra $z\cal{N}$. Since $z \not= 0$ and $\psi$ is faithful, $\psi_z(z)>0$. We define a normal faithful state $\varphi$ on $\mathrm{R}(V)=z\cal{N}$ by
\[
\varphi(x)\ov{\mathrm{def}}{=}\frac{\psi_z(x)}{\psi_z(z)},\quad x\in z\cal{N}.
\]

We now make explicit the identification of the corresponding noncommutative $\L^p$-spaces. First, the central reduction by $z$ induces a canonical complete order isometry
\[
z\L^p(\cal{N},\psi)\to\L^p(z\cal{N},\psi_z)
\]
which preserves the $z\cal{N}$-bimodule structure. Moreover, by the canonical identification between the Haagerup and Kosaki realizations of noncommutative $\L^p$-spaces and by the change of weight theorem, there exists a complete order isometry $
\kappa_p\co\L^p(z\cal{N},\psi_z)\to\L^p(z\cal{N},\varphi)$ which preserves the $z\cal{N}$-bimodule structure. In particular,
\[
\kappa_p\big(ze\L^p(z\cal{N},\psi_z)ze^\perp\big)
=ze\L^p(z\cal{N},\varphi)ze^\perp.
\]
Together with \eqref{corner-central-reduction}, this yields a complete isometry
\[
e\L^p(\cal{N},\psi)e^\perp
\cong
ze\L^p(z\cal{N},\varphi)ze^\perp.
\]
Since $\mathrm{R}(V)=z\cal{N}$ and $V=ze\mathrm{R}(V)ze^\perp$, we obtain $
ze\L^p(z\cal{N},\varphi)ze^\perp
=\L^p(V,\varphi)$ completely isometrically. Hence $
e\L^p(\cal{N},\psi)e^\perp
\cong
\L^p(V,\varphi)$ completely isometrically. Composing this complete isometry with the surjective complete isometry in the second statement gives a surjective complete isometry $\L^p(V,\varphi)\to Y$. Hence the third statement holds.

3. $\Rightarrow$ 2.: Let $V$ be as in the third statement. By the definition of the linking von Neumann algebra, there exists a projection $e \in \mathrm{R}(V)$ such that $V \ov{\eqref{TRO-linking}}{=}e\mathrm{R}(V)e^\perp$. By the definition of the rectangular noncommutative $\L^p$-space, we have
\[
\L^p(V,\varphi)
=e\L^p(\mathrm{R}(V),\varphi)e^\perp
\]
completely isometrically. Taking $\cal{N} \ov{\mathrm{def}}{=} \mathrm{R}(V)$ and $\psi=\varphi$, the surjective complete isometry $T \co \L^p(V,\varphi) \to Y$ gives a surjective complete isometry $e\L^p(\cal{N},\psi)e^\perp\to Y$. Thus the second statement holds, and the proof is complete.
\end{proof}


 

\paragraph{Competing interests} The author declares that he has no competing interests.

\paragraph{Data availability} No data sets were generated during this study.

{\footnotesize

\vspace{0.2cm}

\noindent C\'edric Arhancet\\ 
\noindent 6 rue Didier Daurat, 81000 Albi, France\\
URL: \href{http://sites.google.com/site/cedricarhancet}{https://sites.google.com/site/cedricarhancet}\\
cedric.arhancet@protonmail.com\\
ORCID: 0000-0002-5179-6972 

}

\end{document}